\documentclass[11pt,a4paper]{article}
\usepackage[utf8]{inputenc}
\usepackage{authblk}
\title{\large\textbf{Analytic regularity for Navier-Stokes-Korteweg model on pseudo-measure spaces}}
\author{\small Adrien TENDANI SOLER\footnote{                   \textit{E-mail}: adrien.tendani-soler@math.u-bordeaux.fr}}
\affil{\scriptsize Institut Mathématiques de Bordeaux\\
33405 Talence, France}
\usepackage[nottoc]{tocbibind}
\usepackage[T1]{fontenc}
\usepackage{amsmath}
\usepackage{minitoc}
\usepackage{graphicx}
\usepackage{amsthm}
\usepackage{amsfonts}
\usepackage{amssymb}
\usepackage{varioref}
\usepackage{color}
\usepackage{makeidx}
\usepackage[twoside, inner=3.07cm,outer=2.853cm,top=2.5cm,bottom=3cm]{geometry}
\usepackage{fullpage}
\usepackage{appendix}
\newtheorem{theo}{Theorem}[subsection]

\newtheorem{defi}[theo]{Definition}
\newtheorem{notation}[theo]{Notation}
\newtheorem{lem}[theo]{Lemma}
\newtheorem{prop}[theo]{Proposition}
\newtheorem{coro}[theo]{Corollary}

\DeclareMathOperator{\rad}{rad}

\DeclareMathOperator{\im}{Im}

\newcommand{\C}{\mathbb{C}}
\newcommand{\R}{\mathbb{R}}

\newcommand{\enstq}[2]{\left\{#1~\middle|~#2\right\}}

\DeclareMathOperator{\divergence}{\mathop{}div}
\newcommand*\Laplace{\mathop{}\!\mathbin\bigtriangleup}

\usepackage{abstract}

\newcommand{\revdots}{\mathinner
{ \mkern1mu\raise1pt\vbox{\kern7pt\hbox{.}} \mkern2mu\raise4pt\hbox{.} \mkern2mu\raise7pt\hbox{.}\mkern1mu}
}

\newcommand{\croixdots}{ \mathinner{
\mkern1mu\raise7pt\vbox{\kern7pt 
\hbox{.}}
\mkern-5mu
\raise7pt\vbox{\hbox{.}}
\mkern1mu }
}

\renewcommand{\abstitlestyle}[1]{\noindent}
\date{} 

\begin{document}
\newgeometry{twoside, inner=3.407cm,outer=3.053cm,top=2.23cm,bottom=3.1cm}

{\setlength{\baselineskip}{0.1\baselineskip}
\maketitle}

\hfill
\begin{abstract}
The purpose of this work is to study the existence and analytic smoothing effect for the compressible Navier-Stokes system with quantum pressure in pseudo-measure spaces. This system has been considered by B. Haspot and an analytic smoothing effect for a Korteweg type system was considered by F. Charve, R. Danchin and J. Xu,  both of them in Besov spaces. Here we give a better lower bound of the radius of analyticity near zero.
This work is an opportunity to deepen the study of partial differential equations in pseudo-measure spaces by introducing a new functional setting to deal with non-linear terms. The pseudo-measure spaces are well-adapted to obtain a point-wise control of solutions, with to study of turbulence as perspective.

\end{abstract}
\hfill
{\setlength{\baselineskip}{0.95\baselineskip}
\scriptsize\tableofcontents\par}
\hfill

\section{Introduction}
We are interested by the analytic smoothing properties of the Navier-Stokes-Korteweg system which describe a two-phase compressible and viscous fluids, of density $\rho$ and velocity field $u$. It is generally assumed that the phases are separated by a hypersurface and that the jump in the pressure across the interface is proportional to the curvature. Here we deal with a diffuse interface (DI) model that describes fluids when the change of phase corresponds to a fast but regular transition zone for the density and velocity. This type of models differs from the so-called sharp interface (SI) model when, the interface between phases corresponds to a discontinuity in the state space. The basic ideas of the DI model considering here, is to add to the classical compressible fluids equation a capillary term, that penalizes high variations of the density. The full derivation of the corresponding equation, that we shall name the compressible Navier-Stokes-Korteweg system is due to J. E. Dunn and J. Serrin (see \cite{OntheThermodynamicsofInterstitialWorking}).
\begin{equation}\label{eqNSK}
\begin{cases}
\partial_t\rho+\divergence(\rho u)=0,\\
\partial_t(\rho u)+\divergence(\rho u\otimes u)-{\mathcal A} u+\nabla \Pi=\divergence(\mathcal{K}),\\
(\rho,u)|_{t=0}=(\rho_0,u_0),
\end{cases}
\end{equation}
where $\Pi:=P(\rho)$ is the pressure function, ${\mathcal A}u:=\divergence\left(2\mu(\rho)D_S(u))+\nabla(\nu(\rho)\divergence{u})\right)$ is the diffusion operator, $D_S(u):=\frac 12(\nabla u+\!^t\nabla u)$ is the symmetric gradient and the capillarity tensor is given by
$$
{\mathcal K}:=\rho \divergence(\kappa(\rho)\nabla\rho)I_{\R^d}
+\frac12\big(\kappa(\rho)-\rho\kappa'(\rho)\big)|\nabla\rho|^2I_{\R^d}
-\kappa(\rho)\nabla\rho\otimes\nabla\rho.
$$
This system is due to J. E. Dunn and J. Serrin in \cite{OntheThermodynamicsofInterstitialWorking}.
The density-dependent capillarity function $\kappa$ is assumed to be positive.
Note that for smooth enough density $\rho$ and capillarity function $\kappa$, we have
\begin{equation*}
\divergence{\mathcal K}=\rho\nabla\Big(\kappa(\rho)\Delta\rho+\frac{1}{2}\kappa'(\rho)|\nabla\rho|^2\Big).
\end{equation*}
The coefficients $\nu=\nu(\rho)$ and $\mu=\mu(\rho)$ designate the bulk and shear viscosity, respectively, and are assumed to satisfy in the neighborhood of some reference constant density $\bar\rho>0$ the conditions
\begin{equation*}
\mu>0\qquad\text{and}\qquad \nu+\mu>0.
\end{equation*}
We shall assume that the functions $\lambda,\mu,\kappa$ and $P$ are real analytic in a neighborhood of $\bar\rho$. To simplify, we set $\bar{\rho}=1$.
Introducing $a=\rho-1$ and denoting by $\bar\mu=\mu(1)$, $\bar\nu=\nu(1)$, $\bar\kappa=\kappa(1)$, $\bar \alpha=P'(1)$, the system \eqref{eqNSK} reads
\begin{equation}\label{eqNSK1}
\begin{cases}
\partial_ta+\divergence(u)=\Tilde{f},\\
\partial_tu-\bar{\mathcal A}u+\bar\alpha\nabla a-\bar\kappa\nabla\Delta a=\Tilde{g},
\end{cases}
\end{equation}
where $\bar{\mathcal A}u=2\bar\mu \divergence(D_S(u))+\bar\nu\nabla \divergence u$, $\Tilde{f}=- \divergence(au)$, $\Tilde{g}=\sum_{i=1}^{4}\Tilde{g}_i$ with
\begin{equation*}
\begin{cases}
\Tilde{g}_1:=-u\cdot\nabla u,\\
\displaystyle \Tilde{g}_2:=(1+a)^{-1}{\mathcal A}u-\bar{\mathcal A}u,\\
\displaystyle \Tilde{g}_3:=-(1+a)^{-1}\nabla P(1+a)+\bar \alpha\nabla a,\\
\displaystyle \Tilde{g}_4:=\nabla\Big(\big(\kappa(1+a)-\bar\kappa\big)\Delta a+\frac12\kappa'(1+a)|\nabla a|^2\Big).
\end{cases}
\end{equation*}
The system \eqref{eqNSK1} is a hyperbolic/parabolic coupled system, which is common for compressible Navier-Stokes type systems. In contrast with the linearized equation of the classical compressible Navier-Stokes system, it was remarked by F. Charve, R. Danchin, and J. Xu that for the linear part of \eqref{eqNSK1}, with external forces, both of the density and velocity are smoothed out instantaneously (see lemma \ref{Estimation Duhamel NSK lemme}). In 2018, authors showed in \cite{GevreyanalyticityanddecayforthecompressibleNavier-Stokessystemwithcapillarity} a Gevrey analyticity smoothed effect for all the unknowns of the compressible Navier-Stokes-Korteweg system, in Besov spaces, this is the first related result for a model of compressible fluids. In this paper, we aim to establish this smoothing effect and to estimate the radius of analyticity of the solution, in the pseudo-measure spaces for a particular case presented in the following subsection. Using the method used by J. Y. Chemin, I. Gallagher, and P. Zhang in \cite{Ontheradiusofanalyticityofsolutiontosemi-linearparabolicsystems} for semi-linear parabolic systems, we give a better estimate on the radius of analyticity near 0, the advantage to work in the pseudo-measure spaces is that we obtained point-wise time-frequency estimate of the decay of the solution, with studying the turbulence as perspective. In the following subsection, we describe a special case of the compressible Navier-Stokes-Korteweg system, so-called the incompressible Navier-Stokes system with quantum pressure, that will be discussed in this paper.
\subsection{Compressible Navier-Stokes system with quantum pressure}
In this note, we consider a special case, which is the so-called compressible Navier-Stokes system with quantum pressure considered by B. Haspot \cite{GlobalstrangsolutionfortheKortewegsystemwithquantumpressureindimensionNgeq2}, where
$$(\mu(\rho),\nu(\rho),\kappa(\rho))=(\mu\rho,\nu\rho,\kappa/\rho),\quad P(\rho):=\alpha\rho,$$
and $\mu>0,\mu+\nu>0$, $\kappa>0$, $\alpha>0$ are constants.

Introducing $$\rho=\bar\rho e^a,$$
the system \eqref{eqNSK} reads
\begin{equation}
\label{NSK quantique}
\left \{
\begin{array}{lcr}

     \partial_t a + \divergence(u)=f(u,a),\\ 
      \partial_t u- \mu \Laplace{u}-(\mu+\nu)\nabla\divergence (u) +\alpha\nabla a-\kappa\nabla\Laplace{a}=g(u,a),
     
\end{array}
\right.
\end{equation}
where $g:=\sum_{j=1}^{3}g_{i}$ and  
\begin{equation}
\left \{
\begin{array}{lcr}
  f(u,a) :=-u\cdot \nabla a, \\
  g_1(u,u) :=-u\cdot\nabla u,\\
  g_2(u,a) :=\mu\nabla a\cdot\nabla u+(\mu+\nu)\nabla a\cdot Du,\\
 g_3(a,a) :=\frac{\kappa}{2}\nabla(\nabla a\cdot \nabla a).
\end{array}
\right.
\end{equation}
We consider the initial value condition
\begin{equation}
    \label{NSK quantique condition initial}
    (a,u)_{|_{t=0}}=(a_0,u_0).
\end{equation}
\subsection{Pseudo-measure spaces}
Let us begin by specifying some notations.
\begin{notation}
Throughout the paper, $f\lesssim_{a_1,\dots,a_k}g$ means that there exists a positive constant $C$, which depends on the parameters $a_1,\dots,a_k$ such that $f\leq C g$. We denote by $\widehat{f}$ the Fourier transform with respect to the space variable of the function $f\in\mathcal{C}\left([0,T[;\mathcal{S}'(\R^d)\right)$.
\end{notation}
We begin by define pseudo-measure spaces on the whole space $\R^d$.
For all $r\geq 0$, we define the pseudo-measure space of order $r$ by setting
$$
PM^r(\R^d):=\enstq{g\in\mathcal{S}'(\R^d)}{\widehat{g}\in L^{1}_{loc}(\R^d)\ \, \text{and}\ \, \|g\|_{PM^{r}}:=\sup_{\xi\in\R^d}\lbrace|\xi|^r|\widehat{g}(\xi)|\rbrace <+\infty}.
$$

The pseudo-measure spaces were firstly used for fluids mechanic systems by Y. Le Jan and A.A.S. Sznitman in \cite{CascadesaleatoiresetequationsdeNavier-Stokes} for the incompressible Navier-Stokes system, for existence results. After, the analytic regularity was studying by P. G. Lemarié-Rieusset in \cite{Uneremarque} and W. Deng, M. Paicu and P. Zhang in  \cite{RemarksonthedecayofFouriercoefficientstosolutionofNavier-Stokessystem} for the global mild solution of incompressible Navier-Stokes system. The introduction of pseudo-measure spaces is motivated by \cite{FrequencydecayforNavierStokesstationarysolutions}, related to the theory of turbulences (see also \cite{OntheKolmogorovDissipationLawinaDampedNavierStokesEquation} and \cite{BoundsonKolmogorovspectrafortheNavierStokesequations}). These spaces are particular case of homogeneous Besov spaces construct over the shift-invariant Banach space of distributions. Here, the so-called shift-invariant Banach space of distributions is the pseudo-measure space $PM^0$ (see \cite{RecentdevelopmentsintheNavier-Stokesproblem} for more details).

\subsection{Critical space}
We supposed that $d\geq 2$. 
Here, we want to investigate the existence and regularity for the Cauchy problem associated to \eqref{NSK quantique} in critical spaces, related to the invariance by scaling. The invariance by scaling is the main thread for finding some appropriate functional framework. Let us first recall the notion of scaling for the system \eqref{NSK quantique} (see \cite{ExistenceofsolutionsforcompressiblefluidmodelsofKortewegtype} or \cite{GlobalstrangsolutionfortheKortewegsystemwithquantumpressureindimensionNgeq2} ). If $(a,u)$ solves \eqref{NSK quantique}, then  so does $(a_{\lambda},u_{\lambda})$, where
$$
a_{\lambda}:=a(\lambda^2 \cdot,\lambda \cdot)\ \ \ \text{and}\ \ \ \ u_{\lambda}:=\lambda u(\lambda^2 \cdot,\lambda \cdot),
$$
and $\lambda\in\R^{*}$.  This observation leads to the notion of critical spaces. We say that a functional space is a \textit{critical space} for \eqref{NSK quantique} if for all positive real numbers $\lambda$, the associated norm is invariant under the transformation
$$
(a,u)\longmapsto (a_{\lambda},u_{\lambda}),
$$
up to a constant independent of $\lambda$.  This suggests to choose initial data $(a_0,u_0)$ in the space whose norm is invariant for all positive real number $\lambda$ by $(a_0,u_0)\mapsto (a_0(\lambda\cdot),\lambda u_0(\lambda\cdot))$. If we deal with a pseudo-measure space, a natural candidate is the space $PM^{d}\times PM^{d-1}$. According to the discussion of the critical spaces, the space-time functional space that we investigate in this paper is the following Kato space.
\begin{defi}
Let $p$, $r$ and $T$ be positive real numbers. We define the Kato space $K^{p,r}_{T}$ as the space of $u\in \mathcal{C}_{b}(]0,T];PM^{r+\frac{2}{p}})$ such that the quantity 
$$
\|u\|_{K^{p,r}_{T}}:=\sup_{t\in]0,T]}\{t^{\frac{1}{p}}\|u(t)\|_{PM^{r+\frac{2}{p}}}\},
$$
is finite. We also define the space $K^{p,r}_{\infty}$ of $u\in \mathcal{C}_{b}(]0,+\infty[;PM^{r+\frac{2}{p}})$ such that  
$$
\|u\|_{K^{p,r}_{\infty}}:=\sup_{t\in]0,+\infty[}\{t^{\frac{1}{p}}\|u(t)\|_{PM^{r+\frac{2}{p}}}\},
$$
is finite.
\end{defi}
We observe that the space $K^{p,d}_{\infty}\times K_{\infty}^{p,d-1}$ verifies the invariance by scaling. The Kato spaces is useful to establish Kato types theorems (see \cite{FourierAnalysisandNonlinearPartialDifferentialEquations} and \cite{Ontheradiusofanalyticityofsolutiontosemi-linearparabolicsystems}), such as theorem \ref{theoreme de type Kato pour pseudo-mesure dans l'espace entier}. We use this Kato spaces to establish global existence and regularity results.

\subsection{Radius of analyticity}
If $\Omega$ is an open subset of $\C^d$, we denote by $\mathcal{H}(\Omega)$ the set of holomorphic functions over $\Omega$.
Let $r<d$. If $u\in PM^r(\R^d)$, we define the radius of analyticity of $u$ by setting 
$$
\rad (u):=\sup \enstq{\sigma>0}{ e^{\sigma|D|}u\in PM^r(\R^d)}.
$$
If $u=(u_1,u_2,\dots,u_d)\in \left(PM^r(\R^d)\right)^d$ is a vector field, we define this radius by setting $\rad(u):=\min_{k\in[\![1,d]\!]}\{\rad(u_k)\}$.
For every $\sigma>0$, we denote by $S_{\sigma}$ the open connected set of all $z$ in  $\C^d$ such that $|\im(z)|<\sigma$. The following proposition justifies the denomination "radius of analyticity".
\begin{prop}
\label{Equivalence entre decroissance exponentielle et analycite}
Let $r<d$ and $\sigma>0$. Let $u$ be in $PM^{r}(\R^d)$. If $e^{\sigma|D|}u\in PM^{r}(\R^d)$, then $u$ extends to an unique holomorphic function $U$ in $\mathcal{H}(S_{\sigma})$.
\end{prop}

This proposition means that we can express $u\in PM^r(\R^d)$, whose Fourier transform have an exponential decay, as the trace on $\R^d$ of a function which is holomorphic on some strip  $S_{\sigma}$.

\subsection{Main results}
We recall that $d\geq 2$. Let's assume that $p>2$ is such that $d-3+\frac{4}{p}>0$. This condition  ensures that nonlinear terms are well defined.
We introduce the space $X_{T}$ of $(a,u)\in (K_{T}^{p,d-1}\cap K_{T}^{p,d})\times K_{T}^{p,d-1}$, that we equip with the norm defined by 
$$
\|(a,u)\|_{X_T}:=\max\{\|a\|_{K_{T}^{p,d-1}},\|a\|_{K_{T}^{p,d}}\} +\|u\|_{K_{T}^{p,d-1}}.
$$
Using the language of mild solutions of the Navier-Stokes-Korteweg system, as in \cite{Gevreyregularityfornonlinearanalyticparabolicequations}, we prove the global existence and regularity of the solution to \eqref{NSK quantique} which we state as follows (summing up theorem \ref{**Existence global NSK pseudo-measure**} and theorem \ref{**Analyticite solution global de NSK pseudo-mesure**}).
\begin{theo}
\label{**main theorem 1**}
Given an initial data $(a_0,u_0)$ in $\left(PM^{d-1}\times PM^{d}\right)\times PM^{d-1}$. If $\|(a_0,|D|a_0,u_0)\|_{PM^{d-1}}$ is small enough, then the Cauchy problem \eqref{NSK quantique}-\eqref{NSK quantique condition initial} has a global solution $(a,u)$ in the space $X_{\infty}$ which space analytic at any positive time. Moreover, for any time $t>0$, we have
$$
\rad(a(t),u(t))\geq c_0\sqrt{t},
$$
for some positive constant $c_0$ which depends only on $\nu$, $\mu$, $\kappa$ and $\alpha$.
\end{theo}
The first observation is that the lower bound of the radius of analyticity is similar to \cite{Gevreyregularityfornonlinearanalyticparabolicequations} in the case of Besov spaces. Moreover this regularity result holds for critical initial data. \\
In section 5, we investigate the instantaneous analytic smoothing effect of the system \eqref{NSK quantique}. The following theorem sums up two main results of section 5.
\begin{theo}
Let $\delta$ be in $]0,\frac{2}{p}]$. Let $(a_0,u_0)$ be in $\left(PM^{d-1+\delta}\cap PM^{d+\delta}\right)\times PM^{d-1+\delta}$ an initial data. There exist a positive time $T$ and an unique solution $(a,u)$ in $X_T$ to the Cauchy problem \eqref{NSK quantique}\eqref{NSK quantique condition initial}. Moreover, if $\delta=\frac{2}{p}$, we have
$$
\liminf_{t\rightarrow 0^{+}}\frac{\rad(a(t),u(t))}{\sqrt{t|\ln{(C_1 t)}|}}\geq C_2,
$$
for some positive constants $C_1$ and $C_2$.
\end{theo}
The main interest of this theorem is the amelioration of the improvement radius of analyticity near $0$, proposed by F. Charve, R. Danchin and J. Xu in \cite{Gevreyregularityfornonlinearanalyticparabolicequations}. This result adapts to our framework the new method of J.-Y. Chemin, I. Gallagher and P. Zhang in \cite{Ontheradiusofanalyticityofsolutiontosemi-linearparabolicsystems} to estimate the radius of analyticity near $0$ of the solution to semi-linear parabolic system.
Note compared with theorem \ref{**main theorem 1**}, that this theorem contains a local in time existence and uniqueness result for supercritical initial data and holds for arbitrary large initial data. Additionally, we remark that the constant $C_1$ and the existence time interval, depend on the norm of the initial data (see theorem \ref{**Estimation du rayon d'analyticite pres de 0**} and theorem \ref{theoreme de type Kato pour pseudo-mesure dans l'espace entier}). 
\section{The linearized system}
\subsection{Parabolic estimate for the linearized system}
In this section we investigate the linearized system around $(u,a)=(0,0)$. This system reads
\begin{equation}
\label{NSKL quantique }
\left \{
\begin{array}{lcr}

     \partial_t a + \divergence(u)=F,\\ 
      \partial_t u- \mu \Laplace{u}-(\mu+\nu)\nabla\divergence (u) +\alpha\nabla a-\kappa\nabla\Laplace{a}=G,
     
\end{array}
\right.
\end{equation}
where $F$ and $G$ are externals forces assumed to be, smooth enough. 
For all $\xi\in\R^d$, we define $(d+1)\times(d+1)$-matrix

$$
A(\xi):=\begin{pmatrix}
0 &i\xi_1&\dots&\dots &\dots& i\xi_d\\
i(\alpha\xi_1+\kappa\xi_1|\xi|^2) &\mu|\xi|^2+(\mu+\nu)\xi_{1}^{2}&& &&(\mu+\nu)\xi_{1}\xi_d\\ \vdots&\vdots&\ddots&&\revdots&\vdots\\
\vdots&\vdots &&\croixdots&& \vdots\\ \vdots&\vdots&\revdots&&\ddots&\vdots\\
i(\alpha\xi_d+\kappa\xi_d|\xi|^2)&(\mu+\nu)\xi_{d}\xi_1&\dots&\dots &\dots&\mu|\xi|^2+(\mu+\nu)\xi_{d}^{2} 
\end{pmatrix}
$$

The matrix-valued symbol $A$ is the symbol of the space derivative operator of the linearized system. 
For all $t\geq 0$, we define
$$
W(t):=e^{tA(D)}.
$$
The family of Fourier multipliers $(W(t))_{t\geq 0}$ is the semi-group of the linearized system.
The key point of our study of the Navier-Stokes-Korteweg system with a quantum pressure, is a point-wise estimate of the semi-group $(W(t))_{t\geq 0}$ (that can be found in \cite{GevreyanalyticityanddecayforthecompressibleNavier-Stokessystemwithcapillarity} and \cite{ExistenceofsolutionsforcompressiblefluidmodelsofKortewegtype}). More precisely, we observe that the linear part of system \eqref{NSK quantique} has a parabolic behavior.
\begin{lem}
\label{Estimation Duhamel NSK lemme}
There exists a positive constant $c_0$, depending only on $(\kappa,\mu)$, such that the following inequality holds for all $\xi\in\R^d$ and $t\geq 0$:
\begin{equation}
    \label{Estimation Duhamel NSK}
    |(\widehat{a},|\xi|\widehat{a},\widehat{u})|(t,\xi)\lesssim_{\kappa} e^{-c_0t|\xi|^2}|(\widehat{a},|\xi|\widehat{a},\widehat{u})|(0,\xi)+\int_{0}^{t}e^{-c_0|\xi|^2(t-\tau)}|(\widehat{f},|\xi|\widehat{F},\widehat{G})|(\tau,\xi)|d\tau.
\end{equation}
\end{lem}
This lemma gives a "parabolic decay"  of Fourier modes, in order to obtain the analytic regularisation. This "transfer of parabolicity" is a remarkable property of Korteweg typs model for compressible fluids. 
\section{Global existence}
\subsection{Nonlinear estimates}
In this section we will establish some bilinear estimate, which will be used to control the nonlinear terms of system \eqref{NSK quantique}. We begin by an elementary lemma where we investigate a convolution inequality.
\begin{lem}
\label{Estimation de convolution pour 1/|.|^alpha et 1/|.|^beta Lemme}
Let $d\geq 2$. Let $\alpha$ and $\beta$ be two real numbers such that $\alpha<d$, $\beta<d$ and $\alpha+\beta>d$. Then, for all $\xi\in\R^d$,
\begin{equation}
    \label{Estimation de convolution pour 1/|.|^alpha et 1/|.|^beta}
    \int_{\R^d}\frac{1}{|\xi-\eta|^{\alpha}}\frac{1}{|\eta|^{\beta}}d\eta\lesssim_{\alpha,\beta,d}\frac{1}{|\xi|^{\alpha+\beta-d}}.
\end{equation}
\end{lem}
This estimate will be useful when we consider the product of two functions in pseudo-measure spaces.
\begin{proof}
If $\xi=0$, it is classical that the left-hand side of \eqref{Estimation de convolution pour 1/|.|^alpha et 1/|.|^beta} is infinite as the right-hand side this inequality. Supposed that $\xi\neq 0$. We set 
\begin{align*}
\int_{\R^d}\frac{1}{|\xi-\eta|^{\alpha}}\frac{1}{|\eta|^{\beta}}d\eta=\underbrace{\int_{B(\xi,|\xi|/2)}\frac{1}{|\xi-\eta|^{\alpha}}\frac{1}{|\eta|^{\beta}}d\eta}_{=:I_1} + \underbrace{\int_{B(0,|\xi|/2)}\frac{1}{|\xi-\eta|^{\alpha}}\frac{1}{|\eta|^{\beta}}d\eta}_{=:I_2}\\ + \underbrace{\int_{\R^d\backslash X_{\xi}}\frac{1}{|\xi-\eta|^{\alpha}}\frac{1}{|\eta|^{\beta}}d\eta}_{=:I_3}.
\end{align*}
where  $X_{\xi}:=B(0,|\xi|/2)\cup B(\xi,\frac{|\xi|}{2})$
We only need to estimate $I_1$, $I_2$ and $I_3$. For the  first one, let use begin by remarking that, if $|\xi-\eta|\leq \frac{|\xi|}{2}$, then, using the inverse triangular inequality, we have $|\eta|\geq\frac{|\xi|}{2}$. Therefore, we get
\begin{equation}
    \label{inegalité (2) sur I1}
    I_1\leq\int_{B(\xi,|\xi|/2)}\frac{1}{|\xi-\eta|^{\alpha}}d\eta\left(\frac{2}{|\xi|}\right)^{\beta}.
\end{equation}
We aim to estimate the first factor to the right hand side of \eqref{inegalité (2) sur I1}. Using the change of variables $\zeta\mapsto\zeta+\xi$, we get
$$
\int_{B(\xi,|\xi|/2)}\frac{1}{|\xi-\eta|^{\alpha}}d\eta=\int_{B(0,|\xi|/2)}\frac{1}{|\zeta|^{\alpha}}d\zeta.
$$
Considering the hypothesis $\alpha<d$, we get, using polar coordinates
\begin{equation}
    \label{inégalité (3) sur I1}
    \int_{B(0,|\xi|/2)}\frac{1}{|\zeta|^{\alpha}}d\zeta\lesssim_{\alpha,\beta,d}
\frac{1}{|\xi|^{\alpha-d}}.
\end{equation}
Using \eqref{inégalité (3) sur I1} to estimate the first factor of the right of \eqref{inegalité (2) sur I1}, we obtain
$$
I_1\lesssim_{\alpha,\beta,d}\frac{1}{|\xi|^{\alpha+\beta-d}}.
$$
Observing that $|\eta|\leq\frac{|\xi|}{2}$ implies $|\xi-\eta|\geq\frac{|\xi|}{2}$ and taking into account that $\beta<d$, from the inverse triangular inequality and using polar coordinates, we get as the same way
\begin{align*}
    I_2\lesssim_{\alpha,\beta,d}\int_{B(0,|\xi|/2)}\frac{1}{|\eta|^{\beta}}d\eta\left(\frac{2}{|\xi|}\right)^{\alpha}\lesssim\frac{1}{|\xi|^{\alpha+\beta-d}}.
\end{align*}
We decompose the last term, namely $I_3$, in two part 
$$
I_3=\int_{B(0,3|\xi|/2)\backslash X_{\xi}}\frac{1}{|\xi-\eta|^{\alpha}|\eta|^{\beta}}d\eta+\int_{\R^d\backslash B(0,\frac{3|\xi|}{2})}\frac{1}{|\xi-\eta|^{\alpha}|\eta|^{\beta}}d\eta.
$$For the first one, we have
$$
\int_{B(0,3|\xi|/2)\backslash X_{\xi}}\frac{1}{|\xi-\eta|^{\alpha}|\eta|^{\beta}}d\eta\lesssim_{\alpha,\beta,d}\frac{1}{|\xi|^{\alpha+\beta}}\int_{B(0,\frac{3|\xi|}{2})}d\eta\lesssim_{\alpha,\beta,d}\frac{1}{|\xi|^{\alpha+\beta-d}}.
$$
Observing that, if $\eta\in \R^d\backslash X_{\xi}$, then $|\xi-\eta|\geq\frac{|\eta|}{2}$ and using polar coordinates and the hypothesis $\alpha+\beta>d$, we obtain 
\begin{align*}
    \int_{\R^d\backslash B(0,3|\xi|/2)}\frac{1}{|\xi-\eta|^{\alpha}|\eta|^{\beta}}d\eta\lesssim_{\alpha,\beta,d}\int_{\R^d\backslash B(0,3|\xi|/2)}\frac{1}{|\eta|^{\alpha+\beta}}d\eta\lesssim_{\alpha,\beta,d}\frac{1}{|\xi|^{\alpha+\beta-d}},
\end{align*}
that concludes the proof.
\end{proof}
The lemma above give a point-wise estimate for the decay rate of the convolution, which is the base for considering products in the pseudo-measure spaces. As a consequence of lemma \ref{Estimation de convolution pour 1/|.|^alpha et 1/|.|^beta Lemme}, we get following bilinear estimates.
\begin{lem}
\label{Estimation de la norme de Kato-pseudomesure des nonlin }
Let $a$ and $b$ two homogeneous Fourier multipliers of degree 1. Let $\delta>0$ and let $p>2$. Then, there exists a positive constant $C_b$, that depends on $\delta$, $p$, $d$ and $b$ such that, for every $u$ and $v$ in $K^{p,d-1}_{T}$, we have
\begin{equation}
\label{produit beta}
\|\int_{0}^{t}e^{\delta(t-s)\Laplace}\beta(u,v)(s)ds\|_{K^{p,d-1}_{T}}\leq C_{b}  \|u\|_{K^{p,d-1}_{T}}\|v\|_{K^{p,d-1}_{T}},
\end{equation}
where $\beta(u,v):=b(D)(u\cdot v)$.
 If $p$ additionally satisfies $d-3+\frac{4}{p}>0$, then there exists a positive constant $C_a$, that depends of $\delta$, $p$, $d$ and $a$ such that, for every $u$ and $v$ in $K^{p,d-1}_{T}$, we have
 \begin{equation}
\label{produit alpha}
\|\int_{0}^{t}e^{\delta(t-s)\Laplace}\alpha(u,v)(s)ds\|_{K^{p,d-1}_{T}}\leq  C_{a} \|u\|_{K^{p,d-1}_{T}}\|v\|_{K^{p,d-1}_{T}},
\end{equation}
where $\alpha (u,v):=u\cdot a(D)v$.

\end{lem}

\begin{proof}
By applying the lemma \ref{Estimation de convolution pour 1/|.|^alpha et 1/|.|^beta Lemme} with $\alpha=d-1+\frac{2}{p}$ and $\beta=d-2+\frac{2}{p}$
\begin{align*}
    |\int_{0}^{t} e^{-\delta(t-s)|\xi|^2}\widehat{\alpha(u,v)}(s,\xi)ds|
& \lesssim\int_{0}^{t}e^{-\delta(t-s)|\xi|^2}(|\widehat{u}|\star|\widehat{a(D)v}|)(s,\xi)ds,\\
& \lesssim\int_{\R^d}\int_{0}^{t}e^{-\delta(t-s)|\xi|^2}|\widehat{u}(s,\xi-\eta)||\eta||\widehat{v}(s,\eta)|dsd\eta,\\
&\lesssim \left(\int_{\R^d}\frac{d\eta}{|\xi-\eta|^{d-1 +\frac{2}{p}}|\eta|^{d-2+\frac{2}{p}}}\right)\int_{0}^{t}\frac{e^{-\delta(t-s)|\xi|^2}}{s^{\frac{2}{p}}}ds\\
&\ \ \ \ \ \  \times \|u\|_{K^{p,d-1}_{t}}\|v\|_{K^{p,d-1}_{t}},\\
&\lesssim \frac{1}{|\xi|^{d-3+\frac{4}{p}}}\left(\int_{0}^{t}\frac{e^{-\delta(t-s)|\xi|^2}}{s^{\frac{2}{p}}}ds\right)\|u\|_{K^{p,d-1}_{T}}\|v\|_{K^{p,d-1}_{T}}.
\end{align*}

Using the lemma \ref{Estimation de convolution pour 1/|.|^alpha et 1/|.|^beta Lemme} with $\alpha=\beta=d-1+\frac{2}{p}$, we obtain by the same approach

\begin{align*}
    |\int_{0}^{t} e^{-\delta(t-s)|\xi|^2}\widehat{\beta(u,v)}(s,\xi)ds|
& \lesssim\int_{0}^{t}e^{-\delta(t-s)|\xi|^2}|b(\xi)|(|\widehat{u}|\star|\widehat{v}|)(s,\xi)ds,\\
& \lesssim|\xi|\int_{\R^d}\int_{0}^{t}e^{-\delta(t-s)|\xi|^2}|\widehat{u}(s,\xi-\eta)||\widehat{v}(s,\eta)|dsd\eta,\\
&\lesssim \left(|\xi|\int_{\R^d}\frac{d\eta}{|\xi-\eta|^{d-1 +\frac{2}{p}}|\eta|^{d-1+\frac{2}{p}}}\right)\int_{0}^{t}\frac{e^{-\delta(t-s)|\xi|^2}}{s^{\frac{2}{p}}}ds\\
&\ \ \ \ \ \ \ \times\|u\|_{K^{p,d-1}_{T}}\|v\|_{K^{p,d-1}_{T}},\\
&\lesssim \frac{1}{|\xi|^{d-3+\frac{4}{p}}}\left(\int_{0}^{t}\frac{e^{-\delta(t-s)|\xi|^2}}{s^{\frac{2}{p}}}ds\right)\|u\|_{K^{p,d-1}_{T}}\|v\|_{K^{p,d-1}_{T}}.
\end{align*}

Finally, using that the function $y\in \R_{+}\mapsto e^{-\delta y}y^{1-\frac{1}{p}}$ is bounded and the change of variable $\sigma\mapsto t\sigma$ to make appear (taking into account the hypothesis $p>2$) the beta function, we have
\begin{align*}
 \frac{1}{|\xi|^{d-3+\frac{4}{p}}}\left(\int_{0}^{t}\frac{e^{-\delta(t-s)|\xi|^2}}{s^{\frac{2}{p}}}ds\right),
&=\frac{1}{|\xi|^{d-1+\frac{2}{p}}}\left(\int_{0}^{t}\frac{e^{-c_0(t-s)|\xi|^2}(\delta(t-s)|\xi|^2)^{1-\frac{1}{p}}}{\delta^{1-\frac{1}{p}}(t-s)^{1-\frac{1}{p}}s^{\frac{2}{p}}}ds\right),\\
&\lesssim \frac{1}{\delta^{1-\frac{1}{p}}|\xi|^{d-1+\frac{2}{p}}}\left(\int_{0}^{t}\frac{1}{(t-s)^{1-\frac{1}{p}}s^{\frac{2}{p}}}ds\right),\\
&\lesssim\frac{t^{-\frac{1}{p}}}{\delta^{1-\frac{1}{p}}|\xi|^{d-1+\frac{2}{p}}}\left(\int_{0}^{1}\frac{1}{(1-\sigma)^{1-\frac{1}{p}}\sigma^{\frac{2}{p}}}ds\right),\\
&\lesssim \frac{t^{-\frac{1}{p}}}{\delta^{1-\frac{1}{p}}|\xi|^{d-1+\frac{2}{p}}}.
\end{align*}
This concludes the proof of the lemma.
\end{proof}
For the remainder of this paper, we supposed that $d\geq 2$ and $p>2$ is such that $d-3+\frac{4}{p}>0$.
\subsection{Global existence theorem}
In this section we study the global existence of the solution to system \eqref{NSK quantique} for critical initial data. The main result of this section is the following theorem.
\begin{theo}
\label{**Existence global NSK pseudo-measure**}
There exists $\rho>0$ and $R>0$ such that, for every $(a_0,u_0)$ in $(PM^{d-1}\cap PM^{d})\times PM^{d-1}$ satisfying
$$
\|(a_0,|D|a_{0},u_{0})\|_{PM^{d-1}}\leq \rho,
$$
there exist an unique solution $(a,u)\in X_{\infty}$ of the Cauchy problem \eqref{NSK quantique}\eqref{NSK quantique condition initial}, such that
$$
\|(a,u)\|_{X_{\infty}}\leq R.
$$
\end{theo}
The proof is based on the Banach fixed-point theorem.
\begin{proof}
First observe that for all $v\in PM^{d-1}$, since the function $y\in\R_{+}\mapsto e^{-c_0y}y^{\frac{1}{p}}$ is bounded by $1$ (because $\frac{1}{p}<1$), we have
$$
e^{-c_0t|\xi|^2}t^{\frac{1}{p}}|\widehat{v}(\xi)||\xi|^{d-1+\frac{2}{p}}=\frac{e^{-c_0t|\xi|^2}(c_0t|\xi|^2)^{\frac{1}{p}}|\widehat{v}(\xi)||\xi|^{d-1}}{c_{0}^{\frac{1}{p}}}\leq\frac{1}{c_{0}^{\frac{1}{p}}}\|v\|_{PM^{d-1}},
$$
hence
\begin{equation}
\label{truc de fin}
 \|e^{c_0t\Laplace}v\|_{K_{\infty}^{p,d-1}}\leq \frac{1}{c_{0}^{\frac{1}{p}}}\|v\|_{PM^{d-1}}.
\end{equation}
It follows from lemma \ref{Estimation Duhamel NSK lemme} and \eqref{truc de fin} that, for all $(a_0,u_0)\in (PM^{d-1}\cap PM^{d})\times PM^{d-1}$, we have
$$
W(\cdot)(a_0,u_0)\in K_{\infty}^{p,d-1}
$$
and
\begin{equation}
\label{Inegualité sur le semigroupe de NSKl dans les pseudomesures }
\|W(t)(a_0,u_0)\|_{X_{\infty}}\leq \Tilde{C}\|(a_0,|D|a_{0},u_0)\|_{PM^{d-1}}.
\end{equation}
where $\Tilde{C}$ is a positive constant, that depends only on $c_0$, $p$ and $d$.
Combining lemma \ref{Estimation Duhamel NSK lemme} and the lemma \ref{Estimation de la norme de Kato-pseudomesure des nonlin }, we get the following estimates : for every $a$ and $b$ in $K_{T}^{p,d-1}\cap K_{T}^{p,d}$ and for all $u$ and $v$ in $K_{T}^{p,d-1}$, 
\begin{align}   
    \|\int_{0}^{t}W(t-s)f(u,a)(s)ds\|_{K_{\infty}^{p,d-1}} & \leq C_f\|u\|_{K_{\infty}^{p,d-1}}\|a\|_{K_{\infty}^{p,d-1}},\label{Estimation semigroupe NSKl pour f}\\
    \|\int_{0}^{t}W(t-s)|D|f(u,a)(s)ds\|_{K_{\infty}^{p,d-1}} & \leq \Tilde{C_{f}}\|u\|_{K_{\infty}^{p,d-1}}\|a\|_{K_{\infty}^{p,d}},\label{Estimation semigroupe NSKl pour |D|f}\\   
    \|\int_{0}^{t}W(t-s)g_1(u,v)(s)ds\|_{K_{\infty}^{p,d-1}} & \leq C_{g_1}\|u\|_{K_{\infty}^{p,d-1}}\|v\|_{K_{\infty}^{p,d-1}},\label{Estimation semigroupe NSKl pour g1}\\  
    \|\int_{0}^{t}W(t-s)g_2(u,a)(s)ds\|_{K_{\infty}^{p,d-1}} & \leq C_{g_2}\|u\|_{K_{\infty}^{p,d-1}}\|a\|_{K_{\infty}^{p,d}},\label{Estimation semigroupe NSKl pour g2}\\
    \|\int_{0}^{t}W(t-s)g_3(a,b)(s)ds\|_{K_{\infty}^{p,d-1}} & \leq C_{g_3}\|a\|_{K_{T}^{p,d}}\|b\|_{K_{\infty}^{p,d}},\label{Estimation semigroupe NSKl pour g3}
\end{align}
where positive constants $C_f$, $\Tilde{C}_f$, $C_{g_1}$, $C_{g_2}$ and $C_{g_3}$ only depends of $d$, $p$ and $c_0$.
For any positive real numbers $R$, we denote by $B(0,R)$ the ball of center 0 and radius $R$ in $X_T$.
$$
\begin{array}{rcl}
\Phi: X_{\infty} &\to& X_{\infty}\\
 (a,u) &\mapsto & W(\cdot)(a_0,u_0)+\int_{0}^{\cdot}W(\cdot-s)(f(u,a)(s),g(u,a)(s))ds,
\end{array}
$$ 
where $(a_0,u_0)\in (PM^{d-1}\cap PM^{d})\times K^{d-1}$ are the initial data. 
The goal is to prove the existence of a fixed point for $\Phi$. We assume that 
\begin{equation}
    \label{Petitesse de la norme de (a0,|D|a0,u0) pour NSK}
    \|(a_0,|D|a_0,u_0)\|_{PM^{d-1}}<\rho,
\end{equation}
for $\rho>0$, small enough, which we will be fixed later.
We begin by proving that for some radius $R>0$, small enough, the ball $B(0,R)$ is stable by $\Phi$. If $(a,u)$ is in $B(0,R)$, then, we deduce from \eqref{Inegualité sur le semigroupe de NSKl dans les pseudomesures }, \eqref{Estimation semigroupe NSKl pour f}, \eqref{Estimation semigroupe NSKl pour |D|f}, \eqref{Estimation semigroupe NSKl pour g1}, \eqref{Estimation semigroupe NSKl pour g2}, \eqref{Estimation semigroupe NSKl pour g3} and \eqref{Petitesse de la norme de (a0,|D|a0,u0) pour NSK} that
\begin{equation}
\label{ majoration norme de Phi(a,u) }
    \|\Phi(a,u)\|_{X_{\infty}}\leq 5K_{\Phi}R^2.
\end{equation}
where $K_{\Phi}:=\max\{C_{f},\Tilde{C}_f, C_{g_1}, C_{g_2}, C_{g_3}\}$. Now, we assume that $R$ satisfies 
\begin{equation}
    \label{stabilité de la boul par Phi}
    5K_{\Phi}R<\frac{1}{2},
\end{equation}
and we set
\begin{equation}
    \label{formule pour epsilon}
    \rho:=\frac{R}{2\Tilde{C}}.
\end{equation}
For $R>0$ satisfying \eqref{stabilité de la boul par Phi} and for this choice of $\rho$, using \eqref{ majoration norme de Phi(a,u) } and \eqref{Inegualité sur le semigroupe de NSKl dans les pseudomesures }, we get for all $(a,u)$ in $B(0,R)$
\[
\|\Phi(a,u)\|_{X_{\infty}}<R,
\]
which means that $B(0,R)$ is stable by $\Phi$.
Let $(a,u)$ and $(b,v)$ be in $X_{\infty}$. We have
\begin{align*}
    f(u,a)-f(v,b) & = f(u,a-b)+ f(u-v,b),\\
    g_{1}(u,u)-g_{1}(v,v) & = g_{1}(u-v,u)+g_{1}(v,u-v),\\
    g_{2}(u,a)-g_{2}(v,b) & = g_{2}(u,a-b)+g_{2}(u-v,b),\\
    g_{3}(a,a)-g_{3}(b,b) & = g_{3}(a-b,a)+g_{3}(b,a-b).
\end{align*}
Thus , applying \eqref{Estimation semigroupe NSKl pour f}, \eqref{Estimation semigroupe NSKl pour |D|f}, \eqref{Estimation semigroupe NSKl pour g1}, \eqref{Estimation semigroupe NSKl pour g2} and \eqref{Estimation semigroupe NSKl pour g3}, we deduce, from the triangular inequality, that 

\begin{align*}
    \|\Phi(a,u)-\Phi(b,v)\|_{X_{\infty}} & \leq K_{\Phi}(  \|u-v\|_{X_{\infty}}(\|b\|_{X_{\infty}}+\|u\|_{X_{\infty}}+\|v\|_{X_{\infty}}+2\||D|b\|_{X_{\infty}})\\
    & +\|a-b\|_{X_{\infty}}\|u\|_{X_{\infty}}\\ & +\||D|(a-b)\|_{X_{\infty}}(2\|u\|_{X_{\infty}}+\||D|a\|_{X_{\infty}}+\||D|b\|_{X_{\infty}}))\\
    & \leq 8K_{\Phi}\|(a,u)-(b,v)\|_{X_{\infty}}(\|(a,u)\|_{X_{\infty}}+\|(b,v)\|_{X_{\infty}}).
\end{align*}

Now, if we take $(a,u)$ and $(b,v)$ in the ball $B(0,R)$, from previous inequalities, it follows that
$$
\|\Phi(a,u)-\Phi(b,v)\|_{X_{\infty}}\leq 16K_{\Phi}R\|(a,u)-(b,v)\|_{X_{\infty}}.
$$
However, from \eqref{stabilité de la boul par Phi}, we get
\begin{equation}
    \label{condition de petitesse sur R pour contraction}
    16RK_{\Phi}<1.
\end{equation}
 Now, assume that $R$ satisfies \eqref{condition de petitesse sur R pour contraction}. Since,  \eqref{condition de petitesse sur R pour contraction} implies \eqref{stabilité de la boul par Phi} , then, for $\rho$ given by \eqref{formule pour epsilon}, $\Phi$ is a strict contractive map of $B(0,R)$ into itself. We conclude with Banach fixed-point theorem (see \cite{FunctionalAnalysisSobolevSpacesandPartialDifferentialEquations}, theorem 5.7).
\end{proof}
As a by-product of the proof of this theorem, we obtain the following result.
\begin{coro}
Let $T>0$. Let $(a_0,u_0)$ be in $(PM^{d-1}\cap PM^{d})\times PM^{d-1}$.There exists two positive constants $C_1$ and $C_2$ that depends only of $\mu$,$\nu$, $p$ and $d$ such that, for all solutions $(a,u)$ of \eqref{NSK quantique} in $X_T$, we have
$$
\|(a,u)\|_{X_{T}}\leq C_1\|(a_0,|D|a_0,u_0)\|_{PM^{d-1}}+C_2\|(a,u)\|_{X_T}^{2}.
$$
\end{coro}
\section{Analyticity for global solution}
The purpose of this section is to prove the analyticity of some global solution constructed in the previous section. Furthermore, we give an lower bound of the radius of analyticity. This result holds for small enough critical initial data. We investigate later the case of supercritical data.
\begin{theo}
\label{**Analyticite solution global de NSK pseudo-mesure**}
Let $\rho$ and $R$ as in theorem \ref{**Existence global NSK pseudo-measure**}. For every $(a_0,u_0)$ in the space $(PM^{d-1}\cap PM^{d})\times PM^{d-1}$ such that
$$
\|(a_0,|D|a_0,u_0)\|_{PM^{d-1}}\leq \frac{\rho}{2e^{2c_0}},
$$
the solution of the Cauchy problem \eqref{NSK quantique}-\eqref{NSK quantique condition initial} given by the theorem \ref{**Existence global NSK pseudo-measure**} is analytic in space for every time $t>0$ with a radius of analyticity bounded below by $c_0\sqrt{t}$.
\end{theo}
The proof of the  global existence for analytic solutions follows the main scheme than the proof of global existence. The difference is the choice of the functional space where we look for the fixed point. The idea is to consider a weighted norm of the form $e^{\delta\sqrt{t}|D|}$, where $\delta\sqrt{t}$ gives a radius of analyticity for the solution at any positive time $t$. This method is well-known (see \cite{GevreyanalyticityanddecayforthecompressibleNavier-Stokessystemwithcapillarity} for this system). We begin by a version of lemma \ref{Estimation de la norme de Kato-pseudomesure des nonlin }, with analytic norm. We recall that $d\geq 2$ and $p>2$ is such that $d-3+\frac{4}{p}>0$.
\begin{lem}
\label{Estimation de la norme de Kato-pseudomesure analytique des nonlin }
Let $T$ be in $]0,+\infty]$. Let $a$ and $b$ two homogeneous Fourier multipliers of degree 1. Then for every $u$ and $v$ in the Kato space $K^{p,d-1}_{T}$, we have
\begin{equation}
\label{produit alpha analytique}
\|\int_{0}^{t}e^{c_0(t-s)\Laplace}A_t(u,v)(s)ds\|_{K^{p,d-1}_{T}}\leq  2^{1-\frac{1}{p}}e^{2c_0} C_{a} \|e^{c_0\sqrt{t}|D|}u\|_{K^{p,d-1}_{T}}\|e^{c_0\sqrt{t}|D|}v\|_{K^{p,d-1}_{T}},
\end{equation}
\begin{equation}
\label{produit beta analytique}
\|\int_{0}^{t}e^{c_0(t-s)\Laplace}B_t(u,v)(s)ds\|_{K^{p,d-1}_{T}}\leq  2^{1-\frac{1}{p}}e^{2c_0}C_{b}  \|e^{c_0\sqrt{t}|D|}u\|_{K^{p,d-1}_{T}}\|e^{c_0\sqrt{t}|D|}v\|_{K^{p,d-1}_{T}},
\end{equation}
where, for all $t>0$, $A_t(u,v):=e^{c_0\sqrt{t}|D|}\left(u\cdot a(D)v\right)$ and $B_t(u,v):=e^{c_0\sqrt{t}|D|}\left(b(D)(u\cdot v)\right)$.
\end{lem}
\begin{proof}
We adapt the proof of the lemma \ref{Estimation de la norme de Kato-pseudomesure des nonlin }. The additional key point that we need here is the inequality 
\begin{equation}
    \label{Ineguality convolution poids exponentiel}
    e^{-\frac{c_0}{2}(t-s)|\xi|^2}e^{-c_0\sqrt{s}|\xi-\eta|}e^{-c_0\sqrt{s}|\eta|}\leq e^{2c_0}e^{-c_0\sqrt{t}|\xi|}.
\end{equation}
From the inverse triangular inequality, it follows that $-\sqrt{s}|\xi-\eta|-\sqrt{s}|\eta|\leq -\sqrt{s}|\xi|$. Hence, for establish \eqref{Ineguality convolution poids exponentiel}, it is enough to prove that
$$
I:=(\sqrt{t}-\sqrt{s})|\xi|(1-(\sqrt{t}+\sqrt{s})\frac{|\xi|}{2})\leq 2.
$$
If $\sqrt{t}|\xi|\geq 2$, we deduce that $I\leq 0\leq 2$ and if $\sqrt{t}|\xi|<2$, then $I\leq \sqrt{t}|\xi|\leq 2$. Then we obtain the expected upper bound for $I$, that concludes the proof of the lemma.
\end{proof}
In lemma \ref{Estimation de la norme de Kato-pseudomesure analytique des nonlin }, constants $C_{a}$ and $C_{b}$ are the same as in lemma \ref{Estimation de la norme de Kato-pseudomesure des nonlin }. We can also notice the presence of a factor $2^{1-\frac{1}{p}}$ unlike the non analytic version.
To deal with the analytic setting, we introduce the analytic space $Y_{T}$ of $(a,u)\in (K_{T}^{p,d-1}\cap K_{T}^{p,d})\times K_{T}^{p,d-1}$, that we equip with the norm defined by
$$
\|(a,u)\|_{Y_T}:=\max\{\|e^{c_0\sqrt{t}|D|}a\|_{K_{T}^{p,d-1}},\|e^{c_0\sqrt{t}|D|}a\|_{K_{T}^{p,d}}\} +\|e^{c_0\sqrt{t}|D|}u\|_{K_{T}^{p,d-1}}.
$$

\begin{proof}[Proof of theorem \ref{**Analyticite solution global de NSK pseudo-mesure**}]
First, we remark that for every $\xi\in\R^d$ and positive time $t$, we have
\begin{equation}
\label{Truc egualite}
e^{-c_0t|\xi|^2}(t|\xi|^2)^{\frac{1}{p}}e^{c_0\sqrt{t}|\xi|}=\left(\frac{2}{c_{0}}\right)^{\frac{1}{p}}\times\left(e^{-\frac{c_0}{2}t|\xi|^2}\left(\frac{c_0}{2}t|\xi|^2\right)^{\frac{1}{p}}\right)\times \left(e^{c_0\sqrt{t}|\xi|}e^{-\frac{c_0}{2}t|\xi|^2}\right).
\end{equation}
Since $\frac{1}{p}<1$, the second factor of the right-hand side of the previous identity is lower than $1$. Using the inequality \eqref{Ineguality convolution poids exponentiel}, the third factor to the right-hand of \eqref{Truc egualite} is increased by $e^{2c_0}$. Hence, for all $v\in PM^{d-1}$, we have
$$
e^{-c_0t|\xi|^2}t^{\frac{1}{p}}e^{c_0\sqrt{t}|\xi|}|\widehat{v}(\xi)||\xi|^{d-1+\frac{2}{p}}\leq\frac{2^{\frac{1}{p}}e^{2c_0}}{c_{0}^{\frac{1}{p}}}\|v\|_{PM^{d-1}}.
$$
We suppose that the initial data $(a_0,u_0)\in PM^{d-1}\cap PM^{d})\times PM^{d-1}$ satisfy
\begin{equation}
    \label{Petitesse de la norme de (a0,|D|a0,u0) pour NSK analytique}
    \|(a_0,|D|a_0,u_0)\|_{PM^{d-1}}<\Tilde{\rho},
\end{equation}
for some positive real number $\Tilde{\rho}$, small enough.
Using the lemma \ref{Estimation de la norme de Kato-pseudomesure analytique des nonlin }, we deduce by the same way of the proof of global existence, that for a radius $\Tilde{R}:=\frac{R}{2^{1-\frac{1}{p}}e^{2c_0}}$
and for $\Tilde{\rho}:=\frac{\rho}{2e^{2c_0}}$. The map
$$
\begin{array}{rcl}
\Psi: Y_{\infty} &\to& Y_{\infty}\\
 (a,u) &\mapsto & W(\cdot)(a_0,u_0)+\int_{0}^{\cdot}W(\cdot-s)(f(u,a)(s),g(u,a)(s))ds,
\end{array}
$$ 
have a unique fixed-point $(a,u)$ in the ball of center $0$ and radius $\Tilde{R}$ in $Y_{\infty}$. \\

Furthermore, if the initial data $(a_0,u_0)$ satisfy \eqref{Petitesse de la norme de (a0,|D|a0,u0) pour NSK analytique}, using  $\Tilde{\rho}<\rho$, we deduce the existence of a global solution, constructed as the unique\footnote{We recall that the fixed-point of the Banach fixed-point, like in \cite{FunctionalAnalysisSobolevSpacesandPartialDifferentialEquations} theorem 5.7, is unique} fixed point of $\Phi$ in the ball $B(0,R)$ of $X_{\infty}$. Moreover the fixed-point of $\Psi$ found previously, namely $(a,u)$, is in the ball $B(0,R)$ of the space $X_{\infty}$ and is the unique fixed-point of $\Phi$ in this ball. Indeed,
it is enough to observe that 
$$
\|\Phi(a,u)\|_{X_{\infty}}\leq\|\Psi(a,u)\|_{Y_{\infty}}\leq \Tilde{R}<R,
$$
keeping in mind that $\Tilde{\rho}<\rho$. In particular, we conclude that, if the initial data $(a_0,u_0)$ satisfy \eqref{Petitesse de la norme de (a0,|D|a0,u0) pour NSK analytique}, the solution of theorem \ref{**Existence global NSK pseudo-measure**} is analytic and, at any positive time $t$, its radius of analyticity is bounded below by $c_0\sqrt{t}$.
\end{proof}

\section{Estimate near 0}
Now, we turn our attention to the case of supercritical initial data. First we give a local in time existence and uniqueness theorem (in the subsection 5.1), so-called Kato type theorem. This result holds for supercritical initial data, wich we will pick arbitrarily large. In the subsection 5.2, we establish the analyticity of the solution with supercritical initial data and give the same lower bound as in the case of critical initial data. To end this section, we investigate in the subsection 5.3 the instantaneous smoothing effect, giving a better estimate of the radius of analyticity near 0. This last result constitutes the main novelty of this paper, related to the study of the radius of analyticity for Navier-Stokes-Korteweg type system.
\subsection{Kato type theorem for local existence with supercritical data}

In order to give a better estimate of the radius of analyticity in the neighbourhood of $0$, we need to obtain a local existence results for supercritical initial data, namely in the space $\left(PM^{d-1+\delta}\cap PM^{d+\delta}\right)\times PM^{d-1+\delta}$. Our goal here is to prove a Kato type theorem to obtain the local existence which will be proved by a Banach fixed point argument.
\begin{theo}
\label{theoreme de type Kato pour pseudo-mesure dans l'espace entier}
Let $\delta$ in $]0,\frac 2p]$. For any initial data $(a_0,u_0)\in \left(PM^{d-1+\delta}\cap PM^{d+\delta}\right)\times PM^{d-1+\delta}$, there exists a positive real number $T$ such that the Cauchy problem \eqref{NSK quantique}\eqref{NSK quantique condition initial} has a unique solution $(a,u)$ in the space $X_{T}$. Moreover, there exists a positive constant $c_{\delta}$, that does not depend on the initial data $(a_0,u_0)$, such that $T\geq c_{\delta}\|(a_0,|D|a_0,u_0)\|_{PM^{d-1+\delta}}^{-\frac{2}{\delta}}$.
\end{theo}
Compared with the global existence (theorem \ref{**Existence global NSK pseudo-measure**}) this theorem give the uniqueness of the solution during the existence time. Furthermore, we don't need any smallness assumptions on the initial data, but the existence interval gets smaller as the norm of the initial data gets bigger.
\begin{proof}
Let $(a_0,u_0)$ be in $\left(PM^{d-1+\delta}\cap PM^{d+\delta}\right)\times PM^{d-1+\delta}$ and $T$ a positive time will be chosen later. Let $v\in PM^{d-1+\delta}$. We deduce from the lemma \ref{Estimation Duhamel NSK lemme} that
\begin{equation}
    \label{Estimation semigroup NSK-chaleur avec PM d-1+ delta }
    \|W(t)v\|_{K_{T}^{p,d-1}}\lesssim \|e^{c_0t\Laplace} v\|_{K^{p,d-1}}.
\end{equation}
Furthermore, since the function $y\in \R_{+}\mapsto e^{-c_0 y}y^{\frac{1}{p}-\frac{\delta}{2}}$ is bounded, for every $t\in[0,T]$ and $\xi\in\R^d$ we have
$$
e^{-c_0t|\xi|}t^{\frac{1}{p}}|\widehat{v}(\xi)||\xi|^{d-1+\frac 2p}\lesssim_{c_0,p,\delta}T^{\frac{\delta}{2}}\|v\|_{PM^{d-1+\delta}},
$$
hence
$$
\|e^{c_0t\Laplace}v\|_{PM^{d-1+\frac{2}{p}}}\lesssim_{c_0,p,\rho} T^{\frac{\delta}{2}}\|v\|_{PM^{d-1+\delta}}.
$$
From \eqref{Estimation semigroup NSK-chaleur avec PM d-1+ delta } and the estimate above, we conclude that there exist a positive constant $\Tilde{C}_{\delta}$, such that
\begin{equation}
    \label{Estimation semigroup NSK avec PM d-1+ delta}
    \|W(t)(a_0,u_0)\|_{X_T}\leq \Tilde{C}_{\delta} T^{\frac{\delta}{2}}\|(a_0,|D|a_0,u_0)\|_{PM^{d-1+\delta}}.
\end{equation}
Now, we consider the map
$$
\begin{array}{rcl}
\Phi: X_{T} &\to& X_{T}\\
 (a,u) &\mapsto & W(\cdot)(a_0,u_0)+\int_{0}^{\cdot}W(\cdot-s)(f(u,a)(s),g(u,a)(s))ds.
\end{array}
$$ 
As in the proof of the global existence theorem, for some positive radius $R$, if $(a,u)$ and $(b,v)$ in $B(0,R)$ we have
\begin{align}
    \|\Phi(a,u)-W(\cdot)(a_0,u_0)\|_{X_T} & \leq 5K_{\Phi}R^2,\label{IV}\\
    \|\Phi(a,u)-\Phi(b,v)\|_{X_T} & \leq 16 K_{\Phi} R\|(a,u)-(b,v)\|_{X_T}.\label{V}  
\end{align}
Assume that the radius $R$ and the time $T$ satisfy
$$
16RK_{\Phi}<1\ \ \ \text{and}\ \ \ T= c_{\delta}\|(a_0,|D|a_0,u_0)\|_{PM^{d-1+\delta}}^{-\frac{2}{\delta}}
$$
where $c_{\delta}:=\left(\frac{R}{2\Tilde{C}_{\delta}}\right)^{\frac{\delta}{2}}$. Then, using \eqref{Estimation semigroup NSK avec PM d-1+ delta}, \eqref{IV} and \eqref{V}, we observe that $\Phi$ is an contraction of the ball $B(0,R)$ of $X_T$ into itself. We conclude the proof by applying the Banach fixed-point theorem. 
\end{proof}
As the previous section, we can get the analyticity of solutions with an estimate of the radius of analyticity.

\begin{theo}
\label{theoreme de type Kato analytic pour pseudo-mesure dans l'espace entier}
Let $\delta$ in $]0,\frac 2p]$. For any initial data $(a_0,u_0)\in \left(PM^{d-1+\delta}\cap PM^{d+\delta}\right)\times PM^{d-1+\delta}$, there exists a positive real number $T$ such that the Cauchy problem \eqref{NSK quantique}\eqref{NSK quantique condition initial} has unique solution $(a,u)$ in the space $X_{T}$, such that for all $t\in]0,T]$, $(a(t),u(t))$ is real analytic with
$$
\rad((a(t),u(t)))\geq c_0\sqrt{t}.
$$

Moreover, there exists a positive constant $d_{\delta}$, that does not depend on the initial data $(a_0,u_0)$, such that $T\geq d_{\delta}\|(a_0,|D|a_0,u_0)\|_{PM^{d-1+\delta}}^{-\frac{2}{\delta}}$.
\end{theo}
\begin{proof}
 Let $(a_0,u_0)$ be in $\left(PM^{d-1+\delta}\cap PM^{d+\delta}\right)\cap PM^{d-1+\delta}$ and $T$ a positive time that will be chosen later. Now, we pick $v$ in  $PM^{d-1+\delta}$. The lemma \ref{Estimation Duhamel NSK lemme} provides
 \begin{equation}
     \label{semi-groupe estimation Kato-analytique}
     \|e^{c_0\sqrt{t}|D|}W(t)v\|_{K_{T}^{p,d-1}}\lesssim\|e^{c_0t\Laplace}e^{c_0\sqrt{t}|D|}v\|_{K_{T}^{p,d-1}}.
 \end{equation}
 Moreover, for all $t\in[0, T]$, we have
 \begin{align*}
     e^{-c_0t|\xi|^2}t^{\frac{1}{p}} e^{c_0\sqrt{t}|\xi|}|\widehat{v}(\xi)||\xi|^{d-1+\frac{2}{p}} & =t^{\frac{\delta}{2}}\left(e^{-\frac{c_0}{2}t|\xi|^2}(t|\xi|^{\frac{1}{p}-\frac{1}{\delta}})\right)\times\left(|\widehat{v}(\xi)||\xi|^{d-1+\delta}\right)\\
     & \times\left(e^{-\frac{c_0}{2}t|\xi|^2}e^{c_0\sqrt{t}|\xi|}\right)\\
     & \leq T^{\frac{\delta}{2}}\left(\frac{c_0}{2}\right)^{\frac{1}{p}-\frac{\delta}{2}}e^{2c_0}\|c\|_{PM^{d-1+\delta}}.
 \end{align*}
Combining the last estimate and \eqref{semi-groupe estimation Kato-analytique}, we get
$$
\|e^{c_0\sqrt{t}|D|}W(t)v\|_{K_{T}^{p,d-1}}\lesssim  T^{\frac{\delta}{2}}\|v\|_{PM^{d-1+\delta}}.
$$
Hence, there exists a positive constant $D_{\delta}$ such that
\begin{equation}
    \label{estimation norm YT du semi-groupe}
    \|W(t)(a_0,u_0)\|_{Y_T}\leq D_{\delta} T^{\delta}\|e^{c_0\sqrt{t}|D|}(a_0,|D|a_0,u_0)\|_{PM^{d-1+\delta}}.
\end{equation}
Now we consider the map
$$
\begin{array}{rcl}
\Phi: Y_{T} &\to& Y_{T}\\
 (a,u) &\mapsto & W(\cdot)(a_0,u_0)+\int_{0}^{\cdot}W(\cdot-s)(f(u,a)(s),g(u,a)(s))ds.
\end{array}
$$ 
Let $R$ be a positive radius that will be chosen later. From the lemma \ref{Estimation de la norme de Kato-pseudomesure analytique des nonlin }, we deduce that there exists a positive constant $K_{\Phi}$ such that, for all $(a,u)$ and $(b,v)$ in the ball $B(0,R)$ of $Y_T$, we have
$$
\|\Phi(a,u)-W(\cdot)(a_0,u_0)\|_{Y_T}  \leq 5K_{\Phi}R^2
$$
and
$$
    \|\Phi(a,u)-\Phi(b,v)\|_{Y_T}  \leq 16 K_{\Phi} R\|(a,u)-(b,v)\|_{Y_T}.
$$
For $R>0$ and $T$ such that
$$
16RK_{\Phi}<1\ \ \ \text{and}\ \ \ \ T=d_{\delta}\|(a_0,|D|a_0,u_0)\|_{PM^{d-1+\delta}}^{-\frac{2}{\delta}},
$$
where $\delta:=\left(\frac{R}{2D_{\delta}}\right)^{\frac{\delta}{2}}$, as in the proof of theorem \ref{theoreme de type Kato pour pseudo-mesure dans l'espace entier}, from the Banach fixed-point theorem we conclude that there exists an unique fixed-point $(a,u)\in Y_T$ of $\Phi$ which solve \eqref{NSK quantique}\eqref{NSK quantique condition initial}. Since $(a,u)\in Y_T$, for all $t$ in $]0,T]$, we get $\rad((a(t),u(t)))\leq c_0\sqrt{t}$, which completes the proof.

\end{proof}
\subsection{Estimate of the radius of analyticity}
In this section we show an improvement of the estimate of the radius of analyticity near $0$. We will adapt the method used by J.-Y. Chemin, I. Gallagher and P. Zhang in \cite{Ontheradiusofanalyticityofsolutiontosemi-linearparabolicsystems} to our context, in order to obtain an estimate of the radius of analyticity when the initial data is in the space $(PM^{p,d-1+\frac{2}{p}}\cap PM^{p,d+\frac{2}{p}})\times  PM^{p,d-1+\frac{2}{p}}$. \\
For $f\in L^{1}_{loc}\left([0,T];\mathcal{S}'(\R^d)\right)$, we set
\begin{equation}
\label{definition de f barre}
\underline{f}(t):=e^{-\frac{\lambda^2}{4(1-\varepsilon)c_0}\frac{t}{T}+\frac{\lambda t}{\sqrt{T}}|D|}f(t)\ \ \ \ \ \ (t\in[0,T])
\end{equation}
For all $T>0$ and for each $t\in [0,T]$ and $\varepsilon>0$, we define the  Fourier multiplier
$$
\theta (t,D,\varepsilon):=-\frac{\lambda^2}{4(1-\varepsilon)c_0}\frac{t}{T}+\lambda\frac{t}{\sqrt{T}}|D|.
$$
In order to study the radius of analyticity, we begin to establish some nonlinear estimate in the new analytic norm provided by $e^{\theta (\cdot,D,\varepsilon)}$.
\begin{lem}
\label{Estimation de la norme de Kato-pseudomesure des nonlin type CGZ}
Let $\varepsilon>0$ and $T>0$. For $\alpha$, $\beta$, $u$ and $v$ as in lemma \ref{Estimation de la norme de Kato-pseudomesure des nonlin }, we have the following inequalities
\begin{align}
    \|\int_{0}^{t}e^{c_0(t-s)\Laplace }\underline{\alpha}_t(u,v)(s)ds\|_{K^{p,d-1}_{T}} & \lesssim_{\varepsilon,c_0}e^{\frac{\lambda^2}{4(1-\varepsilon)c_0}}\|\underline{u}\|_{K^{p,d-1}_{T}}\|\underline{v}\|_{K^{p,d-1}_{T}},\\
    \|\int_{0}^{t}e^{c_0(t-s)\Laplace }\underline{\beta}_t(u,v)(s)ds\|_{K^{p,d-1}_{T}} & \lesssim_{\varepsilon,c_0}e^{\frac{\lambda^2}{4(1-\varepsilon)c_0}}\|\underline{u}\|_{K^{p,d-1}_{T}}\|\underline{v}\|_{K^{p,d-1}_{T}},
\end{align}

where $\underline{\alpha}_{t}:=e^{\theta(t,D,\varepsilon)}\alpha$ and $\underline{\beta}_{t}:=e^{\theta(t,D,\varepsilon)}\beta$.
\end{lem}
\begin{proof}
First, we recall some properties of symbols of $\theta (t,D,\varepsilon)$. For every $t$ and $s$  in $[0,T]$ and for all $\xi$ and $\eta$ in $\R^d$, we have  
\begin{align}
    \theta (t,\xi,\varepsilon) & =\theta(t-s,\xi,\varepsilon)+\theta (s,\xi,\varepsilon),\label{linearite en temps theta}\\
    \theta(t,\xi,\varepsilon)-c_0t|\xi|^2 & =-\frac{\lambda^2}{4(1-\varepsilon)c_0}\frac{t}{T}+\lambda\frac{t}{\sqrt{T}}|\xi|-c_0t|\xi|^2\leq \varepsilon c_0 t|\xi|^2,\label{inegalite carre theta}\\
    \theta(t,\xi,\varepsilon) & \leq\theta(t,\xi-\eta,\varepsilon)+\theta(t,\eta,\varepsilon)+\frac{\lambda^2}{4(1-\varepsilon)c_0}\frac{t}{T}\label{presque sous additivite theta}.
\end{align}
Let $\xi\in\R^d$. Using \eqref{linearite en temps theta} and \eqref{inegalite carre theta}, for all $t\in[0,T]$, we get
\begin{align*}
    |\int_{0}^{t}e^{-c_0(t-s)|\xi|^2}\widehat{\underline{\alpha}_{t}(u,v)}(s,\xi)ds|&\leq|\int_{0}^{t}e^{-c_0(t-s)|\xi|^2+\theta(t-s,\xi,\varepsilon)}\left(e^{\theta(s,\xi,\varepsilon)}\widehat{\alpha(u,v)}(s,\xi)\right)ds|,\\
    & \leq \int_{0}^{t}e^{-\varepsilon c_0(t-s)|\xi|^2}\left(e^{\theta(s,\xi,\varepsilon)}|\widehat{\alpha(u,v)}(s,\xi)|\right)ds.
\end{align*}
Therefore, we deduce from \eqref{presque sous additivite theta} that, for all $s\in [0,T]$, we have
\begin{align*}
    e^{\theta(s,\xi,\varepsilon)}|\widehat{\alpha(u,v)}(s,\xi)| & \lesssim \int e^{\theta(s,\xi,\varepsilon)}|\widehat{u}(s,\xi-\eta)||\eta||\widehat{v}(s,\eta)|d\eta,\\
    &\lesssim \int e^{\frac{\lambda^2}{4(1-\varepsilon)c_0}\frac{s}{T}}|\widehat{\underline{u}}(s,\xi-\eta)||\eta||\widehat{\underline{v}}(s,\eta)|d\eta,\\
    & \lesssim e^{\frac{\lambda^2}{4(1-\varepsilon)c_0}}\int|\widehat{\underline{u}}(s,\xi-\eta)||\eta||\widehat{\underline{v}}(s,\eta)|d\eta,\\
    & \lesssim \frac{e^{\frac{\lambda^2}{4(1-\varepsilon)c_0}}}{s^{\frac{2}{p}}}\int\frac{d\eta}{|\xi-\eta|^{d-1+\frac{2}{p}}|\eta|^{d-1+\frac{2}{p}}}\|\underline{u}\|_{K^{p,d-1}_{T}}\|\underline{v}\|_{K^{p,d-1}_{T}}.
\end{align*}
Using the lemma \ref{Estimation de convolution pour 1/|.|^alpha et 1/|.|^beta Lemme}, we obtain 
\begin{equation}
    \label{exponetielle theta alpha chapeau}
    e^{\theta(s,\xi,\varepsilon)}|\widehat{\alpha(u,v)}(s,\xi)|\lesssim \frac{e^{\frac{\lambda^{2}}{4(1-\varepsilon)c_0}}}{|\xi|^{d-3+\frac{4}{p}}s^{\frac{2}{p}}}\|\underline{u}\|_{K^{p,d-1}_{T}}\|\underline{v}\|_{K^{p,d-1}_{T}}.
\end{equation}
Similarly to the end of the proof of the lemma \ref{Estimation de la norme de Kato-pseudomesure des nonlin }, we deduce from \eqref{exponetielle theta alpha chapeau}, that
\begin{align*}
    |\int_{0}^{t}e^{-c_0(t-s)|\xi|^2}\widehat{\underline{\alpha}_t(u,v)}(s,\xi)ds|& \lesssim e^{\frac{\lambda^2}{4(1-\varepsilon)c_0}}\int_{0}^{t}
\frac{e^{-\varepsilon c_0(t-s)|\xi|^2}}{s^{\frac{2}{p}}|\xi|^{d-3+\frac{4}{p}}}ds\|\underline{u}\|_{K^{p,d-1}_{T}}\|\underline{v}\|_{K^{p,d-1}_{T}}\\
& \lesssim \frac{e^{\frac{\lambda^2}{4(1-\varepsilon)c_0}}}{t^{\frac{1}{p}}|\xi|^{d-1++\frac{2}{p}}}\|\underline{u}\|_{K^{p,d-1}_{T}}\|\underline{v}\|_{K^{p,d-1}_{T}}.
\end{align*}
The first inequality follows. The proof of the second inequality is similar, with some modifications in the same manner as the proof of \eqref{produit beta}.
\end{proof}
As in the proof of theorem \ref{**Existence global NSK pseudo-measure**}, using lemma \ref{Estimation de la norme de Kato-pseudomesure des nonlin type CGZ}, we can prove the following estimate.
\begin{lem}
\label{lemme du type CGZ}
Let $\varepsilon>0$ and let $\delta$ be in $[0,\frac 2p]$. Let $(a_0,u_0)\in (PM^{d-1+\delta}\cap PM^{d+\delta})\times  PM^{d-1+\delta}$ an initial data and $(a,u)\in X_T$ a solution of the Cauchy problem \eqref{NSK quantique}\eqref{NSK quantique condition initial}. There exist two positive constants $C$, which depend only of $\nu$, $\mu$, $\alpha$ and $\kappa$, and $C_{\varepsilon}$, which depend only of $\nu$, $\mu$, $\alpha$, $\kappa$ and $\varepsilon$, such that
\begin{equation}
    \|(\underline{a},\underline{u})\|_{X_T}
\leq C\left(\|e^{\varepsilon c_0 t\Laplace}(a_0, u_0)\|_{X_T}+C_{\varepsilon}e^{\frac{\lambda^2}{4(1-\varepsilon)c_0}}\|(\underline{a},\underline{u})\|_{X_{T}}^{2}\right).
\end{equation}
\end{lem}
The lemma above combined with a bootstrap argument is the key point to prove the following theorem, that is the main result of this section.
\begin{theo}
\label{**Estimation du rayon d'analyticite pres de 0**}
Let $(a_0,u_0)\in (PM^{d-1+\frac{2}{p}}\cap PM^{d+\frac{2}{p}})\times PM^{d-1+\frac{2}{p}}$. If there exists a solution $(a,u)\in X_{T^{\star}}$ of \eqref{NSK quantique}\eqref{NSK quantique condition initial}for some positive times $T^{\star}$, then
\begin{equation}
    \liminf_{T\longrightarrow 0^{+}}\frac{\rad(a(t),u(t))}{\sqrt{T\displaystyle\left\lvert\ln{\left(T\|(a_0,|D|a_0,u_{0})\|_{PM^{d-1+\frac{2}{p}}}^{\frac{1}{p}}\right)}\right\rvert}}\geq \sqrt{\frac{4}{p}}.
\end{equation}
\end{theo}
\begin{proof}
We use a bootstrap argument. Let $\varepsilon>0$. For every $T\in [0,T^{\star}]$, we denote by $H(T)$ the following induction hypothesis, 
\begin{equation}
     \|(\underline{a},\underline{u})\|_{X_T}\leq D_{\varepsilon}e^{-\frac{\lambda_{T}^{2}}{4(1-\varepsilon)c_0}},
\end{equation}
where the positive real number $\lambda_T$ will be chosen later and 
$$
D_{\varepsilon}:=\frac{1}{C_{\varepsilon}4\mu C},
$$
with
\begin{equation}
\label{definition de mu}
    \mu:=\frac{1}{2}\frac{1}{2C+4}.
\end{equation}
If $H(T)$ is satisfying, we deduce from the lemma \ref{lemme du type CGZ} that
$$
\|(\underline{a},\underline{u})\|_{X_T}\leq C\|e^{\varepsilon c_0 t\Laplace}(a_0,u_0)\|_{X_T} +\frac{1}{4\mu}\|(\underline{a},\underline{u})\|_{X_{T}},
$$
that is 
\begin{equation}
    \label{inequality type CGZ pour bootstrap}
    \|(\underline{a},\underline{u})\|_{X_T}\leq\frac{4\mu C}{1-4\mu}\|e^{\varepsilon c_0 t\Laplace}(a_0,u_0)\|_{X_T}.
\end{equation}
For all $T\in [0,T^{\star}]$, we have
$$
\|e^{\varepsilon c_0t\Laplace}(a_0,u_0)\|_{X_T}\leq D_{p,\varepsilon}T^{\frac{1}{p}}\|(a_0,|D|a_0,u_0)\|_{PM^{d-1+\frac{2}{p}}}.
$$
Let us define
$$
T_{\varepsilon}:=\eta_{\varepsilon}\|(a_0,|D|a_0,u_0)\|_{PM^{d-1+\frac{2}{p}}}^{-p},
$$
where
$$
\eta_{\varepsilon}:=\left(\frac{D_{\varepsilon}}{2D_{p,\varepsilon}}\right)^{p}.
$$
Then, for every $T\in[0,T_{\varepsilon}]$, we have
    
$$
2D_{p,\varepsilon}T^{\frac{1}{p}}\|(a_0,|D|a_0,u_0)\|_{PM^{d-1+\frac{2}{p}}}\leq D_{\varepsilon}.
$$
Now, for all $T\in[0,T^{\star}]$, we define the positive real number $\lambda_{T}$ by setting
\begin{equation}
    \label{definition de lambda_T}
    \lambda_{T}^{2}:=\frac{4(1-\varepsilon)}{p}\displaystyle\left\lvert\ln{\left(\frac{\eta_{\varepsilon}}{T\|(a_0,|D|a_0,u_0)\|_{PM^{d-1+\frac{2}{p}}}^{\frac{1}{p}}}\right)}\right\rvert.
\end{equation}
According to \eqref{definition de mu}, we have $\frac{4\mu C}{1-4\mu}<2$. For $T\in[0,T_{\varepsilon}]$, assuming $H(T)$, we deduce from \eqref{inequality type CGZ pour bootstrap}, that
\begin{align*}
    \|(\underline{a},\underline{u})\|_{X_T} & \leq\frac{4\mu C}{1-4\mu} D_{p,\varepsilon}T^{\frac{1}{p}}\|(a_0,|D|a_0,u_0)\|_{PM^{d-1+\frac{2}{p}}},\\
    & < 2D_{p,\varepsilon} T^{\frac{1}{p}}\|(a_0,|D|a_0,u_0)\|_{PM^{d-1+\frac{2}{p}}},\\
    & =D_{\varepsilon}e^{-\frac{\lambda_{T}^{2}}{4(1-\varepsilon)c_0}}.
\end{align*}
This in turn shows that $H(T)$ holds for $T\in[0,T_{\varepsilon}]$. Moreover, for all $T\in[0,T_{\varepsilon}]$, from \eqref{definition de lambda_T}, it follows 
$$
T^{\frac{1}{p}}\|e^{\lambda_T\sqrt{T}|D|}(a(T),|D|a
(T),u(T))\|_{PM^{d-1+\frac{2}{p}}}\leq D_{\varepsilon}.
$$
 Hence, for every $T\in[0,T_{\varepsilon}]$, we have
 $$
 R(T)\geq\sqrt{\frac{4(1-\varepsilon)}{p}T\displaystyle\left\lvert\ln{\left(\frac{\eta_{\varepsilon}}{T\|(a_0,|D|a_0,u_0)\|_{PM^{d-1+\frac{2}{p}}}^{\frac{1}{p}}}\right)}\right\rvert},
 $$
 where $R(T):=\rad (a(t),u(t))$. This shows that
     
 \begin{equation}
     \label{limite inf avec epsilon type CGZ}
     \liminf_{T\rightarrow 0^{+}}\frac{R(T)}{\sqrt{T\displaystyle\left\lvert\ln{\left(T\|(a_0,|D|a_0,u_0)\|_{PM^{d-1+\frac{2}{p}}}^{\frac{1}{p}}\right)}\right\rvert}}\geq \sqrt{\frac{4(1-\varepsilon)}{p}}.
 \end{equation}
 Since \eqref{limite inf avec epsilon type CGZ} holds for $\varepsilon>0$ chosen arbitrarily, the theorem is proved.
\end{proof}
For the case of critical initial data, the proof of theorem $1.3$ $(b)$ of \cite{Ontheradiusofanalyticityofsolutiontosemi-linearparabolicsystems} for the semi-linear parabolic equation cannot be adjusted to our functional framework, due to the point-wise feature of pseudo-measure spaces.

\appendix
\section{Characterization of analyticity with Fourier transform}
In this appendix we prove proposition \ref{Equivalence entre decroissance exponentielle et analycite}.
\begin{prop}
\label{Equivalence entre decroissance exponentielle et analycite}
Let $r<d$ and $\sigma>0$. Let $u$ be in $PM^{r}(\R^d)$. If $e^{\sigma|D|}u\in PM^{r}(\R^d)$, then $u$ extends to a unique holomorphic function $U$ in $\mathcal{H}(S_{\sigma})$.
\end{prop}
\begin{proof}
Since $\widehat{u}\in L^{1}_{loc}(\R^d)$, we deduce that $\widehat{u}$ and $|\widehat{u}|^2$ are integrable on the a neighborhood of $0$ and using that $e^{\sigma|D|}u\in PM^{r}(\R^d)$, it is easy to conclude that $\widehat{u}\in L^1(\R^d)\cap L^2(\R^d)$. Then, for almost every $x\in\R^d$, we have 
\begin{equation}
    \label{equalite u in sa TF invers pp}
    u(x)=\frac{1}{(2\pi)^d}\int_{\R^d}e^{ix\cdot\xi}\widehat{u}(\xi)d\xi.
\end{equation}
We denote by $v(x)$ the right-hand side of this inequality. It is sufficient to prove that the function $x\in\R^d\mapsto v(x)$ extends to a holomorphic function on $S_{\sigma}$. If $\Tilde{\sigma}\in]0,\sigma[$, then for all $z\in S_{\Tilde{\sigma}}$, we have
\begin{align*}
    |e^{iz\cdot\xi}\widehat{u}(\xi)| & \leq e^{|\im(z)||\xi|}|\widehat{u}(\xi)|\\
    & \leq e^{\Tilde{\sigma}|\xi|}|\widehat{u}(\xi)|\\
    & \leq \frac{e^{-(\sigma-\Tilde{\sigma})|\xi|}}{|\xi|^r}\|e^{\sigma|D|}u\|_{PM^r}.
\end{align*}
Using the hypothesis $r<d$, we deduce by a classical argument that the function $\xi\mapsto e^{ix\cdot\xi}\widehat{u}(\xi)$ is in $L^1(\R^d)$. This legitimate, for every $z\in  S_{\sigma}$, the definition of the quantities
$$
U(z):=\frac{1}{(2\pi)^d}\int e^{iz\cdot\xi}\widehat{u}(\xi)d\xi.
$$
From
\begin{equation}
   \label{Pour holomorphi sous le signe some espac entier} 
  |e^{iz\cdot\xi}\widehat{u}(\xi)|\leq \frac{e^{-(\sigma-\Tilde{\sigma})|\xi|}}{|\xi|}\|e^{\sigma|D|}u\|_{PM^r},
\end{equation}
that holds for each $z\in S_{\Tilde{\sigma}}$ and $\xi\in\R^d$, and observing that the right-hand side of \eqref{Pour holomorphi sous le signe some espac entier} defines a $L^{1}(\R^d)$ function that does not depend on $z\in S_{\Tilde{\sigma}}$, we deduce that $U\in\mathcal{H}(S_{\Tilde{\sigma}})$. Since $\Tilde{\sigma}$ is arbitrarily chosen in $]0,\Tilde{\sigma}[$, we deduce that $U$ is holomorphic over $S_{\sigma}$.
\end{proof}

\newpage
\bibliographystyle{siam}
\bibliography{Bibliographie}
\end{document}